\documentclass{amsart}
\usepackage{amsmath, amsfonts, amsthm, enumerate, tikz}

\newtheorem{thm}{Theorem}[section]
\newtheorem{cor}[thm]{Corollary}
\newtheorem{lem}[thm]{Lemma}

\theoremstyle{definition}

\theoremstyle{remark}
\newtheorem*{rem}{Remark} %unnumbered remarks
\theoremstyle{definition}
\newtheorem{ex}[thm]{Example}

\numberwithin{equation}{section}

\begin{document}

\allowdisplaybreaks

\title[Newton-like inequalities for sets of complex numbers]{Families of Newton-like inequalities for sets of self-conjugate complex numbers}


%    author one information
\author{Richard Ellard}
\address{
	Richard Ellard,
	School of Mathematics and Statistics,
	University College Dublin,
	Belfield, Dublin 4, Ireland
}
\email{richardellard@gmail.com}
\thanks{The authors' work was supported by Science Foundation Ireland under Grant 11/RFP.1/MTH/3157.}

%    author two information
\author{Helena \v{S}migoc}
\address{
	Helena \v{S}migoc,
	School of Mathematics and Statistics,
	University College Dublin,
	Belfield, Dublin 4, Ireland
}
\email{helena.smigoc@ucd.ie}

%    \subjclass is required.


\date{February 2016}

%    Abstract is required.
\begin{abstract}
We derive families of Newton-like inequalities involving the elementary symmetric functions of sets of self-conjugate complex numbers in the right half-plane. These are the first known inequalities of this type which are independent of the proximity of the complex numbers to the real axis.
\end{abstract}

\maketitle

\section{Introduction}

The $k$-th elementary symmetric function of the variables $x_1,x_2,\ldots x_n$ is defined by
\begin{align*}
	e_0(x_1,x_2,\ldots x_n) &:= 1, \\
	e_k(x_1,x_2,\ldots x_n) &:= \sum_{1\leq i_1<i_2<\cdots<i_k\leq n} x_{i_1}x_{i_2}\cdots x_{i_k} :\hspace{3mm} k=1,2,\ldots,n.
\end{align*}
It will also be convenient to define $e_k(x_1,x_2,\ldots x_n)=0$ if $k<0$ or $k>n$. In order to state the celebrated Newton's inequalities, it is more convenient to consider the $k$-th \emph{elementary symmetric mean}
\[
	E_k(x_1,x_2,\ldots x_n):=\binom{n}{k}^{-1}e_k(x_1,x_2,\ldots x_n): \hspace{3mm} k=0,1,\ldots,n.
\]
For brevity, we will often write simply $e_k$ or $E_k$ when there is no confusion as to the variables involved.

\begin{thm}\label{thm:Newton'sInequalities}{\bf (Newton's Inequalities)}
	If $\mathcal{X}:=(x_1,x_2,\ldots,x_n)$ is a list of real numbers, then
	\begin{equation}\label{eq:Newton'sInequalities}
		E_k(\mathcal{X})^2\geq E_{k-1}(\mathcal{X})E_{k+1}(\mathcal{X}):\hspace{3mm} k=1,2,\ldots,n-1,
	\end{equation}
	with equality if and only if all of the $x_i$ coincide or both sides vanish.
\end{thm}

Theorem \ref{thm:Newton'sInequalities} is a consequence of a rule stated (without proof) by Newton \cite{Newton} which gives a lower bound on the number of nonreal roots of a real polynomial; however, since Newton did not give a proof of his rule, the proof of Theorem \ref{thm:Newton'sInequalities} is due to MacLaurin \cite{Maclaurin}. For an inductive proof in the case where $x_1,x_2,\ldots x_n$ are nonnegative, see \cite[\S 2.22]{Inequalities}. For a proof by differential calculus in the case where $x_1,x_2,\ldots x_n$ are real, see \cite[\S 4.3]{Inequalities}, or alternatively \cite{Rosset}.

Several reformulations/generalisations of Newton's inequalities have been given over the years, for example in \cite{Whiteley,Menon,Rosset} and more recently in \cite{Niculescu,Simic}. The relationship between Newton's inequalities and matrix spectra have been studied in \cite{Holtz(NewtonInequalities),Johnson(NewtonInequalities)}. Newton-like inequalities for certain families of complex numbers have been studied in \cite{Monov,Xu2008,Xu2009}. 

%Newton-like inequalities for certain families of complex numbers have been studied in \cite{Monov,Xu2008,Xu2009} and the relationship between Newton's inequalities and matrix spectra have been studied in \cite{Holtz(NewtonInequalities),Johnson(NewtonInequalities)}. 

In this paper, we give families of Newton-like inequalities for sets of self-conjugate complex numbers with nonnegative real parts and show that the given inequalities are optimal. These inequalities are of particular interest, since no further conditions on the set of complex numbers under consideration are imposed. In general, a sequence of nonnegative numbers $\{E_k\}$ is said to be \emph{log-concave} if $E_k^2\geq E_{k-1}E_{k+1}$ for all $k$. Therefore, the study of Newton-like inequalities for sets of complex numbers is further motivated by the literature on log-concave sequences (see \cite{Stanley,Brenti}).

Note that (\ref{eq:Newton'sInequalities}) is equivalent to
\[
	e_k(\mathcal{X})^2 \geq \frac{k+1}{k}\frac{n-k+1}{n-k}e_{k-1}(\mathcal{X})e_{k+1}(\mathcal{X}),
\]
which is stronger than
\[
	e_k(\mathcal{X})^2 \geq e_{k-1}(\mathcal{X})e_{k+1}(\mathcal{X}).
\]
It is well-known that (\ref{eq:Newton'sInequalities}) is equivalent to
\begin{equation}\label{eq:GeneralisedNewton'sInequalities}
	E_kE_l\geq E_{k-1}E_{l+1} \hspace{2mm} : \hspace{6mm} 1\leq k\leq l\leq n-1,
\end{equation}
provided $E_1,E_2,\ldots,E_n\geq0$ and the sequence $E_1,E_2,\ldots,E_n$ has no internal zeros, namely if $k<l$, then $E_k,E_l>0$ implies $E_i>0$ for all $k<i<l$. This follows from the fact that
\[
	E_k^2E_{k+1}^2\cdots E_l^2\geq(E_{k-1}E_{k+1})(E_kE_{k+2})\cdots(E_{l-1}E_{l+1}).
\]
In particular, if the $x_i$ are nonnegative, then (\ref{eq:GeneralisedNewton'sInequalities}) holds.

Now suppose $\mathcal{X}:=(x_1,x_2,\ldots,x_n)$ is a list of complex numbers. It is natural to assume that $\mathcal{X}$ is self-conjugate (any complex numbers occur in complex-conjugate pairs), since this ensures that each $e_i(\mathcal{X})$ is a real number. We will also assume that the $x_i$ have nonnegative real parts, since this guarantees that $e_i(\mathcal{X})\geq0$, $i=0,1,\ldots,n$.

In general, Newton's inequalities (\ref{eq:Newton'sInequalities}) do not hold under these assumptions; however, Monov \cite{Monov} showed that a weaker version of Theorem \ref{thm:Newton'sInequalities} does hold. For $0\leq\lambda\leq1$, define the wedge
\[
	\Omega:=\{z\in\mathbb{C}:|\,\mathrm{arg}(z)\,|\leq\cos^{-1}\sqrt{\lambda}\}.
\]

\begin{thm}{\bf\cite{Monov}}\label{thm:Monov}
	Let $\mathcal{X}:=(x_1,x_2,\ldots,x_n)$ be a list of self-conjugate variables in $\Omega$. Then
	\begin{equation}\label{eq:LambdaNewtonInequalities}
		E_k(\mathcal{X})^2\geq \lambda E_{k-1}(\mathcal{X})E_{k+1}(\mathcal{X}):\hspace{3mm} k=1,2,\ldots,n-1.
	\end{equation}
\end{thm}

Theorem \ref{thm:Monov} was generalised by Xu \cite{Xu2008,Xu2009}:

\begin{thm}{\bf\cite{Xu2008,Xu2009}}
	Let $\mathcal{X}:=(x_1,x_2,\ldots,x_n)$ be a list of self-conjugate variables in $\Omega$. Then
	\begin{equation}\label{eq:GeneralisedLambdaNewtonInequalities}
		E_k(\mathcal{X})E_l(\mathcal{X})\geq\lambda E_{k-1}(\mathcal{X})E_{l+1}(\mathcal{X}) \: : \hspace{4mm} 1\leq k\leq l\leq n-1.
	\end{equation}
\end{thm}

The inequalities in (\ref{eq:LambdaNewtonInequalities}) are known as the \emph{$\lambda$-Newton inequalities} and those in (\ref{eq:GeneralisedLambdaNewtonInequalities}) are known as the \emph{generalised $\lambda$-Newton inequalities}.

Note that the strength of the inequalities in (\ref{eq:GeneralisedLambdaNewtonInequalities}) depends on the proximity of the $x_i$ to the real axis, via the parameter $\lambda$. In particular, if the $x_i$ are all real, then (\ref{eq:GeneralisedLambdaNewtonInequalities}) reduces to Newton's inequalities. On the other hand, if any of the $x_i$ are purely imaginary, then (\ref{eq:GeneralisedLambdaNewtonInequalities}) reduces to the trivial inequality $E_kE_l\geq0$. In %Section \ref{sec:NewNewtonLikeInequalities}, 
this paper we develop inequalities of the form
\begin{equation}\label{eq:GeneralFormInequalities}
	e_k(\mathcal{X})e_l(\mathcal{X})\geq C e_{k-h}(\mathcal{X})e_{l+h}(\mathcal{X}),
\end{equation}
where the constant $C$ is independent of $x_1,x_2,\ldots,x_n$. Specifically, we will prove:

\begin{thm}\label{thm:InequalitiesSummary}
	Let $\mathcal{X}:=(x_1,x_2,\ldots,x_n)$ be a self-conjugate list of  complex numbers with nonnegative real parts. Then for all $k\leq l$, the following inequalities hold:
	\begin{enumerate}
		\item[\textup{(i)}] $e_{2k}(\mathcal{X})e_{2l}(\mathcal{X})\geq \frac{(l+1) (\lfloor n/2 \rfloor-k+1)}{k (\lfloor n/2 \rfloor-l)}e_{2k-2}(\mathcal{X})e_{2l+2}(\mathcal{X})$;
		\item[\textup{(ii)}] $e_{2k+1}(\mathcal{X})e_{2l+1}(\mathcal{X})\geq \frac{(l+1) (\lceil n/2 \rceil-k)}{k (\lceil n/2 \rceil-l-1)}e_{2k-1}(\mathcal{X})e_{2l+3}(\mathcal{X})$;
		\item[\textup{(iii)}] $e_{2k-1}(\mathcal{X})e_{2l}(\mathcal{X})\geq e_{2k-2}(\mathcal{X})e_{2l+1}(\mathcal{X})$;
		\item[\textup{(iv)}] $e_{2k}(\mathcal{X})e_{2l+1}(\mathcal{X})\geq e_{2k-1}(\mathcal{X})e_{2l+2}(\mathcal{X})$.
	\end{enumerate}
	Furthermore, if all real numbers in $\mathcal{X}$ appear with even multiplicity, then
	\begin{enumerate}
		\item[\textup{(v)}] $e_{2k}(\mathcal{X})e_{2l}(\mathcal{X})\geq \sqrt{\dfrac{l(n-2k)}{k(n-2l)}}e_{2k-1}(\mathcal{X})e_{2l+1}(\mathcal{X})$.
	\end{enumerate}
\end{thm}

In many cases, the inequalities given in Theorem \ref{thm:InequalitiesSummary} are stronger than the corresponding generalised $\lambda$-Newton inequalities (see Section \ref{sec:Optimality}).

\section{New Newton-like inequalities for complex numbers}\label{sec:NewNewtonLikeInequalities}

The following simple example illustrates that, in some cases, the best-possible constant in (\ref{eq:GeneralFormInequalities}) is $C=0$:

\begin{ex}\label{ex:i}
	Consider the list $\mathcal{X}:=(i,-i,i,-i,\ldots,i,-i)$ of length $2m$. We have
	\begin{gather*}
		e_{2i}(\mathcal{X}) = \binom{m}{i} \hspace{2mm} : \hspace{6mm} i=0,1,\ldots,m,\\
		e_{2i+1}(\mathcal{X}) = 0 \hspace{2mm} : \hspace{6mm} i=0,1,\ldots,m-1.
	\end{gather*}
	This example shows us that if $k$, $l$ and $h$ are all odd, then we are forced to choose $C=0$ in (\ref{eq:GeneralFormInequalities}).
\end{ex}

If $k$ and $l$ have the same parity and $h$ is even, the constant C is best-expressed by normalising the elementary symmetric functions in a new way: let us define
\begin{align*}
	P_{2k}(\mathcal{X})&:=\binom{\lfloor n/2 \rfloor}{k}^{-1}e_{2k}(\mathcal{X}) \: : \hspace{4mm} k=0,1,\ldots,\lfloor n/2 \rfloor, \\
	P_{2k+1}(\mathcal{X})&:=\binom{\lceil n/2 \rceil-1}{k}^{-1}e_{2k+1}(\mathcal{X}) \: : \hspace{4mm} k=0,1,\ldots,\lceil n/2 \rceil-1.
\end{align*}

We will require a lemma which appears as Problem 743 in \cite{PHA}:

\begin{lem}\label{lem:Even/Odd}{\bf\cite{PHA}}
	Suppose that the real parts of all roots of the real polynomial $f(x)=x^n+a_1x^{n-1}+\cdots+a_n$ are nonnegative. Then the roots of the polynomials
	\[
		x^n-a_2x^{n-2}+a_4x^{n-4}-\cdots
	\]
	and
	\[
		a_1x^{n-1}-a_3x^{n-3}+a_5x^{n-5}-\cdots
	\]
	are real and interlace.
\end{lem}

\begin{thm}\label{thm:AllEven/AllOdd}
	Let $\mathcal{X}:=(x_1,x_2,\ldots,x_n)$ be a self-conjugate list of complex numbers with nonnegative real parts. Then
	\begin{equation}\label{eq:EvenStatement}
		P_{2k}(\mathcal{X})P_{2l}(\mathcal{X})\geq P_{2k-2}(\mathcal{X})P_{2l+2}(\mathcal{X}) \: : \hspace{4mm} 1\leq k\leq l\leq \lfloor n/2 \rfloor-1
	\end{equation}
	and
	\begin{equation}\label{eq:OddStatement}
		P_{2k+1}(\mathcal{X})P_{2l+1}(\mathcal{X})\geq P_{2k-1}(\mathcal{X})P_{2l+3}(\mathcal{X}) \: : \hspace{4mm} 1\leq k\leq l\leq \lceil n/2 \rceil-2.
	\end{equation}
\end{thm}
\begin{proof}
	First suppose $n$ is even and write $n=2m$. The polynomial
	\[
		x^{2m}+e_1(\mathcal{X})x^{2m-1}+e_2(\mathcal{X})x^{2m-2}+\cdots+e_{2m}(\mathcal{X})
	\]
	has roots $-x_1,-x_2,\ldots,-x_{2m}$. Therefore, by Lemma \ref{lem:Even/Odd}, the polynomial
	\[
		x^{2m}-e_2(\mathcal{X})x^{2m-2}+e_4(\mathcal{X})x^{2m-4}-\cdots+(-1)^{m}e_{2m}(\mathcal{X})
	\]
	has real roots. Hence the roots of the polynomial
	\[
		w^{m}-e_2(\mathcal{X})w^{m-1}+e_4(\mathcal{X})w^{m-2}-\cdots+(-1)^{m}e_{2m}(\mathcal{X}),
	\]
	say $w_1,w_2,\ldots,w_m$, are real and nonnegative. Setting $\mathcal{W}:=(w_1,w_2,\ldots,w_m)$, we note that $e_k(\mathcal{W})=e_{2k}(\mathcal{X})$: $k=0,1,\ldots,m$, and hence, applying Newton's inequalities (\ref{eq:GeneralisedNewton'sInequalities}) to $\mathcal{W}$ gives
	\[
		\frac{e_{2k}(\mathcal{X})}{\binom{m}{k}}\frac{e_{2l}(\mathcal{X})}{\binom{m}{l}}\geq\frac{e_{2k-2}(\mathcal{X})}{\binom{m}{k-1}}\frac{e_{2l+2}(\mathcal{X})}{\binom{m}{l+1}}
	\]
	or
	\[
		P_{2k}(\mathcal{X})P_{2l}(\mathcal{X})\geq P_{2k-2}(\mathcal{X})P_{2l+2}(\mathcal{X}) \: : \hspace{4mm} 1\leq k\leq l\leq m-1.
	\]
	
	Similarly, by Lemma \ref{lem:Even/Odd}, the polynomial
	\[
		e_1(\mathcal{X})x^{2m-1}-e_3(\mathcal{X})x^{2m-3}+e_5(\mathcal{X})x^{2m-5}-\cdots+(-1)^{m-1}e_{2m-1}(\mathcal{X})x
	\]
	has real roots. If $e_1(\mathcal{X})=0$, then $\mathrm{Re}(x_i)=0$ for all $i$. This would imply that $e_{2k+1}(\mathcal{X})=0$ for all $0\leq k\leq m-1$, in which case (\ref{eq:OddStatement}) holds trivially. If $e_1(\mathcal{X})>0$, it follows that the roots of the polynomial
	\[
		w^{m-1}-\frac{e_3(\mathcal{X})}{e_1(\mathcal{X})}w^{m-2}+\frac{e_5(\mathcal{X})}{e_1(\mathcal{X})}w^{m-3}-\cdots+(-1)^{m-1}\frac{e_{2m-1}(\mathcal{X})}{e_1(\mathcal{X})},
	\]
	say $w_1,w_2,\ldots,w_{m-1}$, are real and nonnegative. We note that
	\[
		e_k(\mathcal{W})=\frac{e_{2k+1}(\mathcal{X})}{e_1(\mathcal{X})} \hspace{2mm} : \hspace{6mm} k=0,1,\ldots,m-1,
	\]
	and hence, applying Newton's inequalities (\ref{eq:GeneralisedNewton'sInequalities}) to $\mathcal{W}$ gives
	\[
		\frac{e_{2k+1}(\mathcal{X})}{\binom{m-1}{k}}\frac{e_{2l+1}(\mathcal{X})}{\binom{m-1}{l}}\geq\frac{e_{2k-1}(\mathcal{X})}{\binom{m-1}{k-1}}\frac{e_{2l+3}(\mathcal{X})}{\binom{m-1}{l+1}}
	\]
	or
	\[
		P_{2k+1}(\mathcal{X})P_{2l+1}(\mathcal{X})\geq P_{2k-1}(\mathcal{X})P_{2l+3}(\mathcal{X}) \: : \hspace{4mm} 1\leq k\leq l\leq m-2.
	\]
	
	The proof for odd $n$ is similar.
\end{proof}

As with Newton's inequalities, we note that (\ref{eq:EvenStatement}) and (\ref{eq:OddStatement}) are stronger than $e_{2k}e_{2l}\geq e_{2k-2}e_{2l+2}$ and $e_{2k+1}e_{2l+1}\geq e_{2k-1}e_{2l+3}$, respectively.

If $k$ and $l$ have different parity and $h=1$, it turns out that the best-possible constant in (\ref{eq:GeneralFormInequalities}) is $C=1$:

\begin{thm}\label{thm:DifferentParity}
	Let $\mathcal{X}:=(x_1,x_2,\ldots,x_n)$ be a list of self-conjugate variables with nonnegative real parts. If $k$ and $l$ have different parity, $1\leq k<l\leq n-1$, then
	\begin{equation}\label{eq:EvenOddGap}
		e_k(\mathcal{X})e_l(\mathcal{X})\geq e_{k-1}(\mathcal{X})e_{l+1}(\mathcal{X}).
	\end{equation}
\end{thm}
\begin{proof}
	Let us write
	\[
		\mathcal{X}=\left( a_1\pm ib_1,a_2\pm ib_2,\ldots,a_m\pm ib_m,\mu_1,\mu_2,\ldots,\mu_s \right),
	\]
	where $n=2m+s$ and the $a_i$, $b_i$ and $\mu_i$ are nonnegative. Consider the functions
	\[
		f(a_1,\ldots,a_m,b_1,\ldots,b_m,\mu_1,\ldots,\mu_s)=e_k(\mathcal{X})e_l(\mathcal{X})-e_{k-1}(\mathcal{X})e_{l+1}(\mathcal{X})
	\]
	and
	\[
		g(a_1,\ldots,a_m,b_1,\ldots,b_m,\mu_1,\ldots,\mu_s)=e_k(\mathcal{X})e_l(\mathcal{X})-e_{k-2}(\mathcal{X})e_{l+2}(\mathcal{X})
	\]
	as multivariable polynomials in $a_1,\ldots,a_m,b_1,\ldots,b_m,\mu_1,\ldots,\mu_s$. We claim that
	\begin{enumerate}[(i)]
		\item for all $1\leq k<l\leq n-1$, where $k$ and $l$ have different parity, the coefficient of every term in $f$ is positive and
		\item for all $2\leq k\leq l\leq n-2$, where $k$ and $l$ have the same parity, the coefficient of every term in $g$ is positive.
	\end{enumerate}
	
	The proof is by induction on $n$. If $n=1$ or $n=2$, then there is nothing to prove. Now assume that (i) and (ii) hold for all lists of length strictly less than $n$. If $s>0$, we note that for $i=0,1,\ldots,n$,
	\[
		e_i(\mathcal{X})=e_i(\mathcal{X}')+\mu_se_{i-1}(\mathcal{X}'),
	\]
	where
	\[
		\mathcal{X}':=\left( a_1\pm ib_1,a_2\pm ib_2,\ldots,a_m\pm ib_m,\mu_1,\mu_2,\ldots,\mu_{s-1} \right).
	\]
	Therefore, we may write
	\[
		e_k(\mathcal{X})e_l(\mathcal{X})-e_{k-1}(\mathcal{X})e_{l+1}(\mathcal{X})=A \mu _s^2+B \mu _s+C,
	\]
	where
	\begin{align*}
		A&:=e_{k-1}(\mathcal{X}') e_{l-1}(\mathcal{X}')-e_{k-2}(\mathcal{X}') e_l(\mathcal{X}'), \\
		B&:=e_k(\mathcal{X}') e_{l-1}(\mathcal{X}')-e_{k-2}(\mathcal{X}') e_{l+1}(\mathcal{X}'), \\
		C&:=e_k(\mathcal{X}') e_l(\mathcal{X}')-e_{k-1}(\mathcal{X}') e_{l+1}(\mathcal{X}').
	\end{align*}
	If $k$ and $l$ have different parity, $1\leq k<l\leq n-1$, then the inductive hypothesis guarantees that $A$, $B$ and $C$ consist entirely of positive terms. Hence every term in $f$ is positive.
	
	Similarly, we may write
	\[
		e_k(\mathcal{X})e_l(\mathcal{X})-e_{k-2}(\mathcal{X})e_{l+2}(\mathcal{X})=A \mu _s^2+B \mu _s+C,
	\]
	where
	\begin{align*}
		A:=\;&e_{k-1}(\mathcal{X}') e_{l-1}(\mathcal{X}')-e_{k-3}(\mathcal{X}') e_{l+1}(\mathcal{X}'), \\
		B:=\;&e_{k-1}(\mathcal{X}') e_l(\mathcal{X}')-e_{k-3}(\mathcal{X}') e_{l+2}(\mathcal{X}')\\
		+\;&e_k(\mathcal{X}') e_{l-1}(\mathcal{X}')-e_{k-2}(\mathcal{X}') e_{l+1}(\mathcal{X}'), \\
		C:=\;&e_k(\mathcal{X}') e_l(\mathcal{X}')-e_{k-2}(\mathcal{X}') e_{l+2}(\mathcal{X}').
	\end{align*}
	If $k$ and $l$ have the same parity, $2\leq k\leq l\leq n-2$, then the inductive hypothesis again guarantees that $A$, $B$ and $C$ consist entirely of positive terms. Hence every term in $g$ is positive.
	
	On the other hand, if $s=0$, we note that for $i=0,1,\ldots,n$,
	\[
		e_i(\mathcal{X})=e_i(\mathcal{X}')+2a_m e_{i-1}(\mathcal{X}')+\left(a_m^2+b_m^2\right)e_{i-2}(\mathcal{X}'),
	\]
	where
	\[
		\mathcal{X}':=\left( a_1\pm ib_1,a_2\pm ib_2,\ldots,a_{m-1}\pm ib_{m-1} \right).
	\]
	Hence, we may write
	\begin{multline*}
		e_k(\mathcal{X})e_l(\mathcal{X})-e_{k-1}(\mathcal{X})e_{l+1}(\mathcal{X}) \\
		=A(a_m^2+b_m^2)^2+(B+2Xa_m)(a_m^2+b_m^2)+4Ya_m^2+2Za_m+C,
	\end{multline*}
	where
	\begin{align*}
		A:=\;&e_{k-2}(\mathcal{X}') e_{l-2}(\mathcal{X}')-e_{k-3}(\mathcal{X}') e_{l-1}(\mathcal{X}'), \\
		B:=\;&e_{k-2}(\mathcal{X}') e_l(\mathcal{X}')-e_{k-3}(\mathcal{X}') e_{l+1}(\mathcal{X}')\\
		+\;&e_k(\mathcal{X}') e_{l-2}(\mathcal{X}')-e_{k-1}(\mathcal{X}') e_{l-1}(\mathcal{X}'), \\
		C:=\;&e_k(\mathcal{X}') e_l(\mathcal{X}')-e_{k-1}(\mathcal{X}') e_{l+1}(\mathcal{X}'), \\
		X:=\;&e_{k-1}(\mathcal{X}') e_{l-2}(\mathcal{X}')-e_{k-3}(\mathcal{X}') e_l(\mathcal{X}'), \\
		Y:=\;&e_{k-1}(\mathcal{X}') e_{l-1}(\mathcal{X}')-e_{k-2}(\mathcal{X}') e_l(\mathcal{X}'), \\
		Z:=\;&e_k(\mathcal{X}') e_{l-1}(\mathcal{X}')-e_{k-2}(\mathcal{X}') e_{l+1}(\mathcal{X}').
	\end{align*}
	If $k$ and $l$ have different parity, $1\leq k<l\leq n-1$, then the inductive hypothesis guarantees that $A$, $B$, $C$, $X$, $Y$ and $Z$ consist entirely of positive terms. Hence every term in $f$ is positive, as before.
	
	Similarly, we may write
	\begin{multline*}
		e_k(\mathcal{X})e_l(\mathcal{X})-e_{k-2}(\mathcal{X})e_{l+2}(\mathcal{X}) \\
		=A(a_m^2+b_m^2)^2+(B+2Xa_m)(a_m^2+b_m^2)+4Ya_m^2+2Za_m+C,
	\end{multline*}
	where
	\begin{align*}
		A:=\;&e_{k-2}(\mathcal{X}') e_{l-2}(\mathcal{X}')-e_{k-4}(\mathcal{X}') e_l(\mathcal{X}'), \\
		B:=\;&e_k(\mathcal{X}') e_{l-2}(\mathcal{X}')-e_{k-4}(\mathcal{X}') e_{l+2}(\mathcal{X}'),\\
		C:=\;&e_k(\mathcal{X}') e_l(\mathcal{X}')-e_{k-2}(\mathcal{X}') e_{l+2}(\mathcal{X}'), \\
		X:=\;&e_{k-2}(\mathcal{X}') e_{l-1}(\mathcal{X}')-e_{k-4}(\mathcal{X}') e_{l+1}(\mathcal{X}') \\
		+\;&e_{k-1}(\mathcal{X}') e_{l-2}(\mathcal{X}')-e_{k-3}(\mathcal{X}') e_l(\mathcal{X}'), \\
		Y:=\;&e_{k-1}(\mathcal{X}') e_{l-1}(\mathcal{X}')-e_{k-3}(\mathcal{X}') e_{l+1}(\mathcal{X}'), \\
		Z:=\;&e_{k-1}(\mathcal{X}') e_l(\mathcal{X}')-e_{k-3}(\mathcal{X}') e_{l+2}(\mathcal{X}') \\
		+\;&e_k(\mathcal{X}') e_{l-1}(\mathcal{X}')-e_{k-2}(\mathcal{X}') e_{l+1}(\mathcal{X}').
	\end{align*}
	If $k$ and $l$ have the same parity, $2\leq k \leq l\leq n-2$, then the inductive hypothesis guarantees that $A$, $B$, $C$, $X$, $Y$ and $Z$ consist entirely of positive terms. Hence every term in $g$ is positive, as before.
\end{proof}

\begin{rem}
	In the proof of Theorem \ref{thm:DifferentParity}, we saw that if $k$ and $l$ have the same parity, then $e_k(\mathcal{X})e_l(\mathcal{X})\geq e_{k-2}(\mathcal{X})e_{l+2}(\mathcal{X})$ and the difference $e_k(\mathcal{X})e_l(\mathcal{X})-e_{k-2}(\mathcal{X})e_{l+2}(\mathcal{X})$ is a multivariable polynomial in $a_1,\ldots,a_m,b_1,\ldots,b_m,\mu_1,\ldots,\mu_s$ consisting entirely of positive terms. This inequality is weaker than the inequality $P_k(\mathcal{X})P_l(\mathcal{X})\geq P_{k-2}(\mathcal{X})P_{l+2}(\mathcal{X})$, obtained from Theorem \ref{thm:AllEven/AllOdd}, but the difference $P_k(\mathcal{X})P_l(\mathcal{X})-P_{k-2}(\mathcal{X})P_{l+2}(\mathcal{X})$ does not consist entirely of positive terms.
\end{rem}

It is clear that if $k$ and $l$ have different parity, then Theorem \ref{thm:DifferentParity} implies
\[
	e_k(\mathcal{X})e_l(\mathcal{X})\geq e_{k-h}(\mathcal{X})e_{l+h}(\mathcal{X}) \: : \hspace{4mm} h=2,3,\ldots;
\]
however, such inequalities may always be strengthened by combining Theorems \ref{thm:AllEven/AllOdd} and \ref{thm:DifferentParity}. For example, if $n$ is odd, then it is clear from the definition of $P_i$ that
\begin{equation}\label{eq:P=e}
	\frac{P_{2k-1}P_{2k+2}}{P_{2k-2}P_{2k+3}}=
	\frac{e_{2k-1}e_{2k+2}}{e_{2k-2}e_{2k+3}}
\end{equation}
and in this case,
\begin{align*}
	P_{2k}P_{2k+1} &\geq \sqrt{P_{2k-2}P_{2k-1}P_{2k+2}P_{2k+3}} \\
	&\geq P_{2k-2}P_{2k+3},
\end{align*}
where the first inequality follows from Theorem \ref{thm:AllEven/AllOdd} and the second follows from Theorem \ref{thm:DifferentParity} and (\ref{eq:P=e}). This is stronger than the inequality $e_{2k}e_{2k+1}\geq e_{2k-2}e_{2k+3}$, which would be obtained from Theorem \ref{thm:DifferentParity} alone.

We have yet to consider the case when $k$ and $l$ are both even in (\ref{eq:GeneralFormInequalities}), but $h$ is odd. Specifically, we ask if it is possible to derive inequalities of form
\[
	e_{2k}(\mathcal{X})e_{2l}(\mathcal{X})\geq C e_{2k-1}(\mathcal{X})e_{2l+1}(\mathcal{X}) \: : \hspace{4mm} 1\leq k\leq l\leq (n-1)/2,
\]
where $C>0$. It turns out that if we allow $\mathcal{X}$ to contain unpaired real numbers, then the answer is negative, as the following example illustrates:

\begin{ex}
	Consider the list
	\[
		\mathcal{X}:=( \underbrace{\epsilon i,-\epsilon i,\epsilon i,-\epsilon i,\ldots,\epsilon i,-\epsilon i}_{m\: \mathrm{pairs}}\,,1 )
	\]
	of length $n=2m+1$. We have
	\begin{gather*}
		e_1(\mathcal{X})=1,\\
		e_{2i}(\mathcal{X})=e_{2i+1}(\mathcal{X})=\binom{m}{i}\epsilon^i \: : \hspace{4mm} i=1,2,\ldots,m.
	\end{gather*}
	Hence, for all $1\leq k\leq l\leq m$,
	\[
		\frac{e_{2k}(\mathcal{X})e_{2l}(\mathcal{X})}{e_{2k-1}(\mathcal{X})e_{2l+1}(\mathcal{X})}=\left( \frac{m+1}{k}-1 \right)\epsilon.
	\]
	This example shows us that, given any $k$, $l$ and $n$, it is always possible to find a list $\mathcal{X}$ of length $n$, such that $e_{2k}(\mathcal{X})e_{2l}(\mathcal{X})$ is arbitrarily small compared to $e_{2k-1}(\mathcal{X})e_{2l+1}(\mathcal{X})$.
\end{ex}

Surprisingly, if we insist that $\mathcal{X}$ contain only complex-conjugate pairs (all real numbers in $\mathcal{X}$ appear with even multiplicity), it turns out that
\begin{equation}\label{eq:e2k2Preview}
	e_{2k}(\mathcal{X})^2\geq e_{2k-1}(\mathcal{X})e_{2k+1}(\mathcal{X}) \: : \hspace{4mm} k=1,2,\ldots,m-1.
\end{equation}
We note the similarity of (\ref{eq:e2k2Preview}) to Newton's inequalities (\ref{eq:Newton'sInequalities}). To prove (\ref{eq:e2k2Preview}), we first require a technical lemma:

\begin{lem}\label{lem:eExpansions}
	Let $\mathcal{X}:=(a_1\pm ib_1,a_2\pm ib_2,\ldots,a_m\pm ib_m)$, where $a_i,b_i\geq0:i=1,2,\ldots,m$. Let $U:=\{1,2,\ldots,m\}$ and for each $S\subseteq U$, let $\mathcal{W}_S:=\left(a_i^2+b_i^2:i\in S\right)$. Then for $0\leq k\leq m$,
	\begin{equation}\label{eq:e2kExpansion}
		e_{2k}(\mathcal{X})=\sum_{r=0}^k2^{2r}\sum_{\scriptsize\begin{array}{c}S\subseteq U\\ |S|=2r\end{array}}\left(\prod_{i\in S}a_i\right)e_{k-r}(\mathcal{W}_{U\setminus S})
	\end{equation}
	and for $1\leq k\leq m$,
	\begin{equation}\label{eq:e2k-1Expansion}
		e_{2k-1}(\mathcal{X})=\sum_{r=0}^{k-1}2^{2r+1}\sum_{\scriptsize\begin{array}{c}S\subseteq U\\ |S|=2r+1\end{array}}\left(\prod_{i\in S}a_i\right)e_{k-r-1}(\mathcal{W}_{U\setminus S}).
	\end{equation}
\end{lem}
\begin{proof}
	The proof is by induction on $m$. If $m=1$, then (\ref{eq:e2kExpansion}) and (\ref{eq:e2k-1Expansion}) give $e_0(\mathcal{X})= 1$, $e_1(\mathcal{X})=2a_1$ and $e_2(\mathcal{X})=a_1^2+b_1^2$, as required. Now assume the statement holds for lists with $m-1$ complex-conjugate pairs.
	
	We note that for $i=0,1,\ldots,2m$,
	\begin{equation}\label{eq:eExpansions1}
		e_i(\mathcal{X})=e_i(\mathcal{X}')+2a_me_{i-1}(\mathcal{X}')+(a_m^2+b_m^2)e_{i-2}(\mathcal{X}'),
	\end{equation}
	where $\mathcal{X}':=(a_1\pm ib_1,a_2\pm ib_2,\ldots,a_{m-1}\pm ib_{m-1})$. Hence, by (\ref{eq:eExpansions1}) and the inductive hypothesis,
	\begin{align}
		e_{2k}(\mathcal{X}) =\; &\sum_{r=0}^k2^{2r}\sum_{\scriptsize\begin{array}{c}S\subseteq U'\\ |S|=2r\end{array}}\left(\prod_{i\in S}a_i\right)e_{k-r}(\mathcal{W}_{U'\setminus S})\notag \\
		+\; &a_m\sum_{r=1}^k2^{2r}\sum_{\scriptsize\begin{array}{c}S\subseteq U'\\ |S|=2r-1\end{array}}\left(\prod_{i\in S}a_i\right)e_{k-r}(\mathcal{W}_{U'\setminus S})\label{eq:eExpansions2} \\
		+\; &(a_m^2+b_m^2)\sum_{r=0}^{k-1}2^{2r}\sum_{\scriptsize\begin{array}{c}S\subseteq U'\\ |S|=2r\end{array}}\left(\prod_{i\in S}a_i\right)e_{k-r-1}(\mathcal{W}_{U'\setminus S}),\notag
	\end{align}
	where $U':=\{1,2,\ldots,m-1\}$; however, since
	\begin{align*}
		\sum_{\scriptsize\begin{array}{c}S\subseteq U\\ |S|=2r\end{array}}\left(\prod_{i\in S}a_i\right)e_{k-r}(\mathcal{W}_{U\setminus S}) =\; &\sum_{\scriptsize\begin{array}{c}S\subseteq U'\\ |S|=2r\end{array}}\left(\prod_{i\in S}a_i\right)e_{k-r}\left(\mathcal{W}_{(U'\setminus S)\cup\{m\}}\right) \\
		+\; &a_m\sum_{\scriptsize\begin{array}{c}S\subseteq U'\\ |S|=2r-1\end{array}}\left(\prod_{i\in S}a_i\right)e_{k-r}(\mathcal{W}_{U'\setminus S})
	\end{align*}
	and
	\[
		e_{k-r}\left(\mathcal{W}_{(U'\setminus S)\cup\{m\}}\right)=e_{k-r}\left(\mathcal{W}_{U'\setminus S}\right)+(a_m^2+b_m^2)e_{k-r-1}\left(\mathcal{W}_{U'\setminus S}\right),
	\]
	it follows that the right hand side of (\ref{eq:eExpansions2}) equals
	\[
		\sum_{r=0}^k2^{2r}\sum_{\scriptsize\begin{array}{c}S\subseteq U\\ |S|=2r\end{array}}\left(\prod_{i\in S}a_i\right)e_{k-r}(\mathcal{W}_{U\setminus S}).
	\]
	This establishes (\ref{eq:e2kExpansion}).
	
	The proof of (\ref{eq:e2k-1Expansion}) is similar.
\end{proof}

\begin{thm}\label{thm:e2k2}
	Let $\mathcal{X}:=(a_1\pm ib_1,a_2\pm ib_2,\ldots,a_m\pm ib_m)$, where $a_i,b_i\geq0:i=1,2,\ldots,m$. Then
	\begin{equation}\label{eq:e2k2}
		e_{2k}(\mathcal{X})^2\geq e_{2k-1}(\mathcal{X})e_{2k+1}(\mathcal{X}) \: : \hspace{4mm} k=1,2,\ldots,m-1.
	\end{equation}
\end{thm}
\begin{proof}
	Let $U:=\{1,2,\ldots,m\}$ and for each $S\subseteq U$, let $\mathcal{V}_S:=\left(a_i^2-b_i^2:i\in S\right)$. We will show that
	\[
		e_{2k}(\mathcal{X})^2-e_{2k-1}(\mathcal{X})e_{2k+1}(\mathcal{X})\geq\Theta,
	\]
	where
	\[
		\Theta:=\sum_{r=0}^{k-1}2^{2r}\sum_{\scriptsize\begin{array}{c}S\subseteq U\\ |S|=r\end{array}}\left(\prod_{i\in S}a_i^2b_i^2\right)e_{k-r}(\mathcal{V}_{U\setminus S})^2.
	\]
	More specifically, consider the function
	\[
		f(a_1,\ldots,a_m,b_1,\ldots,b_m)=e_{2k}(\mathcal{X})^2-e_{2k-1}(\mathcal{X})e_{2k+1}(\mathcal{X})-\Theta
	\]
	as a multivariable polynomial in $a_1,\ldots,a_m,b_1,\ldots,b_m$. We will prove:
	
	\underline{\textsc{Claim 1}:} The coefficient of every term in $f$ is positive.
	
	Ultimately, the proof of Claim 1 will be by induction on $m$; however, before we begin, there is a term in $f$ whose coefficient we must explicitly compute. Consider
	\[
		T:=\prod_{i=1}^ka_i^2b_i^2.
	\]
	
	\underline{\textsc{Claim 2}:} The coefficient of $T$ in $f$ is $2^{2k}$.
	
	In order to prove Claim 2, we will determine the coefficients of $T$ in $e_{2k}^2(\mathcal{X})$, $e_{2k-1}(\mathcal{X})e_{2k+1}(\mathcal{X})$ and $\Theta$ separately. First, recall that, by Lemma \ref{lem:eExpansions}, $e_{2k}(\mathcal{X})$ may be written in the form (\ref{eq:e2kExpansion}). The coefficient of $T$ in $e_{2k}(\mathcal{X})^2$ is calculated by considering the sum $\sum T_1T_2$, where the sum is over all appropriately chosen terms $T_1$ and $T_2$ in (\ref{eq:e2kExpansion}). Suppose $T_1$ and $T_2$ correspond to choices $S=S_1$ and $S=S_2$ in (\ref{eq:e2kExpansion}), respectively. It is clear that since each $a_i$ in $T$ has exponent 2, the only contributions to the coefficient of $T$ in $e_{2k}(\mathcal{X})$ come from choosing $S_1=S_2\subseteq\{1,2,\ldots,k\}$. In fact, we must choose $S_1=S_2=\emptyset$, since $i\in S$ implies $e_{k-r}(\mathcal{W}_{U\setminus S})$ is independent of $b_i$. Hence, the only contributions to the coefficient of $T$ come from setting $r=0$ in (\ref{eq:e2kExpansion}), i.e. the coefficient of $T$ in $e_{2k}(\mathcal{X})^2$ is precisely the coefficient of $T$ in $e_k(\mathcal{W}_U)^2$. This is the same as the coefficient of $T$ in $\prod_{i=1}^k\left(a_i^2+b_i^2\right)^2$, which equals $2^k$.
	
	Similarly, we note that $e_{2k-1}(\mathcal{X})$ my be written in the form (\ref{eq:e2k-1Expansion}) and that
	\begin{equation}\label{eq:e2k+1Expansion}
		e_{2k+1}(\mathcal{X}) = \sum_{r=0}^k2^{2r+1}\sum_{\scriptsize\begin{array}{c}S\subseteq U\\ |S|=2r+1\end{array}}\left(\prod_{i\in S}a_i\right)e_{k-r}(\mathcal{W}_{U\setminus S}).
	\end{equation}
	Since it is not possible to choose $S=\emptyset$ in (\ref{eq:e2k-1Expansion}) or (\ref{eq:e2k+1Expansion}), we conclude that the coefficient of $T$ in $e_{2k-1}(\mathcal{X})e_{2k+1}(\mathcal{X})$ is zero.
	
	To compute the coefficient of $T$ in $\Theta$, we note that for any set of integers $i_1,i_2,\ldots,i_{k-r}$ satisfying $1\leq i_1<i_2<\cdots<i_{k-r}\leq k$ and $i_1,i_2,\ldots,i_{k-r}\not\in S$, the coefficient of
	$
		\prod_{j=1}^{k-r}a_{i_j}^2b_{i_j}^2
	$
	in $e_{k-r}(\mathcal{V}_{U\setminus S})^2$ is simply its coefficient in
	\[
		\prod_{j=1}^{k-r}\left( a_{i_j}^2-b_{i_j}^2 \right)^2,
	\]
	which equals $(-2)^{k-r}$. Hence, the coefficient of $T$ in $\Theta$ is
	\[
		(-2)^k\sum_{r=0}^{k-1}(-2)^r\binom{k}{r}=2^k(1-2^k).
	\]
	This establishes Claim 2.
	
	We are now ready to prove Claim 1 by induction. If $m=2$, we need only check the claim holds for $k=1$. Setting $\mathcal{X}=(a_1\pm ib_1,a_2\pm ib_2)$,
	\begin{multline*}
		e_2(\mathcal{X})^2-e_1(\mathcal{X})e_3(\mathcal{X})-e_1\left(a_1^2-b_1^2,a_2^2-b_2^2\right)^2 =\\
		4 a_1^3 a_2+8 a_1^2 a_2^2+4 a_1 a_2^3+4 a_1^2 b_1^2+4 a_1 a_2 b_1^2+4 a_1 a_2 b_2^2+4 a_2^2 b_2^2.
	\end{multline*}
	
	Now assume the claim holds for all lists with $m-1$ complex-conjugate pairs and all $1\leq k\leq m-2$. Note that for $i=0,1,\ldots,2m$,
	\[
		e_i(\mathcal{X})=e_i(\mathcal{X}')+2a_me_{i-1}(\mathcal{X}')+(a_m^2+b_m^2)e_{i-2}(\mathcal{X}'),
	\]
	where $\mathcal{X}':=(a_1\pm ib_1,a_2\pm ib_2,\ldots,a_{m-1}\pm ib_{m-1})$. Hence, for $k=1,2,\ldots,m-1$,
	\begin{multline}\label{eq:eExpansion}
		e_{2k}(\mathcal{X})^2-e_{2k-1}(\mathcal{X})e_{2k+1}(\mathcal{X})= \\
		A\left(a_m^2+b_m^2\right)^2+\left(B+2X a_m\right)\left(a_m^2+b_m^2\right)+4Y a_m^2+2Z a_m+C,
	\end{multline}
	where
	\begin{align*}
		A&:=e_{2k-2}(\mathcal{X}')^2-e_{2k-3}(\mathcal{X}') e_{2k-1}(\mathcal{X}'), \\
		B&:=-e_{2k-1}(\mathcal{X}')^2+2e_{2k-2}(\mathcal{X}')e_{2k}(\mathcal{X}')-e_{2k-3}(\mathcal{X}') e_{2k+1}(\mathcal{X}'), \\
		C&:=e_{2k}(\mathcal{X}')^2-e_{2k-1}(\mathcal{X}') e_{2k+1}(\mathcal{X}'), \\
		X&:=e_{2k-2}(\mathcal{X}') e_{2k-1}(\mathcal{X}')-e_{2k-3}(\mathcal{X}')e_{2k}(\mathcal{X}'), \\
		Y&:=e_{2k-1}(\mathcal{X}')^2-e_{2k-2}(\mathcal{X}')e_{2k}(\mathcal{X}'), \\
		Z&:=e_{2k-1}(\mathcal{X}')e_{2k}(\mathcal{X}')-e_{2k-2}(\mathcal{X}') e_{2k+1}(\mathcal{X}').
	\end{align*}
	Similarly, we may expand $\Theta$ in terms of $a_m$ and $b_m$:
	\begin{equation}\label{eq:ThetaExpansion}
		\Theta=\alpha \left(a_m^4+b_m^4\right)+2\beta  a_m^2b_m^2+2\gamma \left(a_m^2-b_m^2\right)+\delta,
	\end{equation}
	where
	\begin{align*}
		U' :=& \{1,2,\ldots,m-1\}, \\
		\alpha :=& \sum_{r=0}^{k-1}2^{2r}\sum_{\scriptsize\begin{array}{c}S\subseteq U'\\ |S|=r\end{array}}\left(\prod_{i\in S}a_i^2b_i^2\right)e_{k-r-1}(\mathcal{V}_{U'\setminus S})^2, \\
		\beta :=& \sum_{r=0}^{k-2}2^{2r}\sum_{\scriptsize\begin{array}{c}S\subseteq U'\\ |S|=r\end{array}}\left(\prod_{i\in S}a_i^2b_i^2\right)e_{k-r-1}(\mathcal{V}_{U'\setminus S})^2 \\
		&-2^{2(k-1)}\sum_{\scriptsize\begin{array}{c}S\subseteq U'\\ |S|={k-1}\end{array}}\prod_{i\in S}a_i^2b_i^2, \\
		\gamma :=& \sum_{r=0}^{k-1}2^{2r}\sum_{\scriptsize\begin{array}{c}S\subseteq U'\\ |S|=r\end{array}}\left(\prod_{i\in S}a_i^2b_i^2\right)e_{k-r}(\mathcal{V}_{U'\setminus S})e_{k-r-1}(\mathcal{V}_{U'\setminus S}), \\
		\delta :=& \sum_{r=0}^{k-1}2^{2r}\sum_{\scriptsize\begin{array}{c}S\subseteq U'\\ |S|=r\end{array}}\left(\prod_{i\in S}a_i^2b_i^2\right)e_{k-r}(\mathcal{V}_{U'\setminus S})^2.
	\end{align*}
	
	Let us first consider the terms in $f$ which are independent of $a_m$ and $b_m$. By (\ref{eq:eExpansion}) and (\ref{eq:ThetaExpansion}), the sum of all such terms is given by $C-\delta$. Hence, the inductive hypothesis guarantees that every such term is positive.
	
	Next, let us consider the terms in $f$ which depend on either $a_m^4$ or $b_m^4$. The sum of all terms in $f$ which depend on $a_m^4$ is given by $(A-\alpha)a_m^4$ and the sum of all terms in $f$ which depend on $b_m^4$ is given by $(A-\alpha)b_m^4$. Observe that $\alpha$ may be written as
	\begin{align*}
		\alpha=&\sum_{r=0}^{k-2}2^{2r}\sum_{\scriptsize\begin{array}{c}S\subseteq U'\\ |S|=r\end{array}}\left(\prod_{i\in S}a_i^2b_i^2\right)e_{k-r-1}(\mathcal{V}_{U'\setminus S})^2 \\
		&+2^{2(k-1)}\sum_{\scriptsize\begin{array}{c}S\subseteq U'\\ |S|=k-1\end{array}}\prod_{i\in S}a_i^2b_i^2.
	\end{align*}
	Hence, by Claim 2, for any $S\subseteq U'$ with $|S|=k-1$, the term
	$
		\prod_{i\in S}a_i^2b_i^2 \hspace{6mm}
	$
	in $(A-\alpha)$ vanishes. The inductive hypothesis guarantees that the coefficients of all other terms in $(A-\alpha)$ are positive. Hence, every term in $f$ which depends on $a_m^4$ or $b_m^4$ is positive.
	
	Similarly, the sum of all terms in $f$ which depend on $a_m^2b_m^2$ is given by $2(A-\beta)$, but since
	\[
		A-\beta=A-\alpha+2^{2k-1}\sum_{\scriptsize\begin{array}{c}S\subseteq U'\\ |S|={k-1}\end{array}}\prod_{i\in S}a_i^2b_i^2,
	\]
	we see that the coefficient of every such term in $f$ is positive.
	
	Now consider those terms in $f$ in which the exponent of $a_m$ is 1 or 3. The sum of all such terms is given by $2Xa_m(a_m^2+b_m^2)+2Za_m$. It follows from the proof of Theorem \ref{thm:DifferentParity} that every term in $X$ and $Z$ is positive. Hence every term in $f$ in which the exponent of $a_m$ is 1 or 3 is positive.
	
	By symmetry, we have shown that a given term in $f$ is positive if any of the following conditions are satisfied for any $i=1,2,\ldots,m$:
	\begin{enumerate}[(i)]
		\item it is independent of $a_i$ and $b_i$;
		\item it is of the form $Da_i^4$ or $Db_i^4$, where $D$ is independent of $a_i$ and $b_i$;
		\item it is of the form $Da_i^2b_i^2$, where $D$ is independent of $a_i$ and $b_i$;
		\item it is a term in which the exponent of $a_i$ is 1 or 3.
	\end{enumerate}
	From the expansions given in (\ref{eq:eExpansion}) and (\ref{eq:ThetaExpansion}), we see that a general term in $f$ has the form
	\begin{equation}\label{eq:GeneralfTerm}
		a_1^{\zeta_1}a_2^{\zeta_2}\cdots a_m^{\zeta_m}b_1^{\eta_1}b_2^{\eta_2}\cdots b_m^{\eta_m},
	\end{equation}
	where for each $i=1,2,\ldots,m$, $\zeta_i\in\{0,1,2,3,4\}$, $\eta_i\in\{0,2,4\}$ and $\zeta_i+\eta_i\leq4$. If (\ref{eq:GeneralfTerm}) does not satisfy any of the conditions (i)--(iv) above for any $i$, then
	\[
		(\zeta_i,\eta_i)= (0,2) \hspace{2mm} \text{or} \hspace{2mm} (2,0) \hspace{2mm} : \hspace{6mm} i=1,2,\ldots,m.
	\]
	In particular, the degree of such a term is equal to $2m$; however, since every term in $f$ has degree $4k$, we conclude the following: if $m\neq 2k$, then every term in $f$ must satisfy one of the above conditions for some $i$ and if $m=2k$, then any term which does not satisfy any of the above conditions for any $i$, can, up to relabelling the $(a_i,b_i)$, be written in the form
	\[
		T^*:=\left(\prod _{i=1}^p a_i^2\right)\left(\prod _{i=p+1}^{2k} b_i^2\right)
	\]
	for some $p=0,1,\ldots,2k$. Therefore, it suffices to show:
	
	\underline{\textsc{Claim 3}:} If $m=2k$, the coefficient of $T^*$ in $f$ is nonnegative.
	
	In order to prove Claim 3, we will compute the coefficients of $T^*$ in $e_{2k}^2(\mathcal{X})$, $e_{2k-1}(\mathcal{X})e_{2k+1}(\mathcal{X})$ and $\Theta$ separately. Our logic will be similar to that used in the proof of Claim 2.
	
	Using (\ref{eq:e2kExpansion}) and the fact that the exponent of each $a_i$ in $T^*$ is 2, we conclude that the coefficient of $T^*$ in $e_{2k}(\mathcal{X})^2$ is the same as its coefficient in
	\[
		\sum_{r=0}^{\lfloor p/2 \rfloor}2^{4r}\sum_{\scriptsize\begin{array}{c}S\subseteq \{1,\ldots,p\}\\ |S|=2r\end{array}}\left(\prod_{i\in S}a_i^2\right)e_{k-r}\left(\mathcal{W}_{\{1,\ldots,2k\}\setminus S}\right)^2.
	\]
	In addition, for any $S\subseteq\{1,\ldots,p\}$ with $|S|=2r$, the coefficient of
	\[
		\left( \prod_{i\in\{1,\ldots,p\}\setminus S}a_i^2 \right)\left( \prod_{i=p+1}^{2k}b_i^2 \right)
	\]
	in $e_{k-r}\left(\mathcal{W}_{\{1,\ldots,2k\}\setminus S}\right)^2$ is $\binom{2(k-r)}{k-r}$. To see this, note that for arbitrary subsets $S_1\subseteq\{1,\ldots,p\}\setminus S$ and $S_2\subseteq\{p+1,p+2,\ldots,2k\}$, the coefficient of
	$
		\left( \prod_{i\in S_1}a_i^2 \right)\left( \prod_{i\in S_2}b_i^2 \right)
	$
	in $e_{k-r}\left(\mathcal{W}_{\{1,\ldots,2k\}\setminus S}\right)$ is 1 if $|S_1|+|S_2|=k-r$ and zero otherwise and there are $\binom{2(k-r)}{k-r}$ ways of choosing $S_1$ and $S_2$ subject to $|S_1|+|S_2|=k-r$. It follows that the coefficient of $T^*$ in $e_{2k}(\mathcal{X})^2$ is given by
	\[
		\sum _{r=0}^{\lfloor p/2 \rfloor} 2^{4r}\binom{p}{2r}\binom{2(k-r)}{k-r}.
	\]
	
	Similarly, using (\ref{eq:e2k-1Expansion}) and (\ref{eq:e2k+1Expansion}), one can show that the coefficient of $T^*$ in $e_{2k-1}(\mathcal{X})e_{2k+1}(\mathcal{X})$ is given by
	\[
		4\sum _{r=0}^{\lfloor (p-1)/2 \rfloor} 2^{4r}\binom{p}{2r+1}\binom{2(k-r)-1}{k-r}.
	\]
	
	Next, note that the coefficient of $T^*$ in $\Theta$ is precisely its coefficient in $e_k\left(\mathcal{V}_{\{1,\ldots,2k\}}\right)^2$. Consider arbitrary subsets $S_1\subseteq\{1,2,\ldots,p\}$, $S_2\subseteq\{p+1,p+2,\ldots,2k\}$. If $|S_1|+|S_2|\neq k$, then the coefficient of
	\begin{equation}\label{eq:HalfOfek2}
		\left( \prod_{i\in S_1}a_i^2 \right)\left( \prod_{i\in S_2}b_i^2 \right)
	\end{equation}
	in $e_k\left(\mathcal{V}_{\{1,\ldots,2k\}}\right)$ is zero. If $|S_1|+|S_2|=k$, then the coefficient of (\ref{eq:HalfOfek2}) in $e_k\left(\mathcal{V}_{\{1,\ldots,2k\}}\right)$ is equal to its coefficient in $\prod_{i\in S_1\cup S_2}(a_i^2-b_i^2)$, which equals $(-1)^{|S_2|}$. Since there are $\binom{2k}{k}$ ways of choosing $S_1,S_2$ subject to $|S_1|+|S_2|=k$, it follows that the coefficient of $T^*$ in $\Theta$ is $(-1)^p\binom{2k}{k}$. Therefore, we have shown that the coefficient of $T^*$ in $f$ is
	\begin{multline}\label{eq:T*Coefficient}
		\sum _{r=0}^{\lfloor p/2 \rfloor} 2^{4r}\binom{p}{2r}\binom{2(k-r)}{k-r}  \\
		-4\sum _{r=0}^{\lfloor (p-1)/2 \rfloor} 2^{4r}\binom{p}{2r+1}\binom{2(k-r)-1}{k-r}
		-(-1)^p\binom{2k}{k}.
	\end{multline}
	
	The remainder of the proof is devoted to showing that the quantity given in (\ref{eq:T*Coefficient}) is nonnegative. First suppose that $p$ is odd and write $p=2q+1$. In this case, noting that $\binom{2(k-r)}{k-r}=2\binom{2(k-r)-1}{k-r}$, we may write (\ref{eq:T*Coefficient}) as
	\begin{equation}\label{eq:DEvo2}
		\sum _{r=0}^q 2^{4r}\binom{2(k-r)}{k-r}\left( \binom{2q+1}{2r}-2\binom{2q+1}{2r+1} \right) +\binom{2k}{k}.
	\end{equation}
	If $q=0$, then (\ref{eq:DEvo2}) vanishes, so assume $q\geq1$. At this point, it is helpful to consider two related sums which are explicitly summable:
	\begin{gather}
		\sum _{r=0}^q 2^{2r}\binom{2q+1}{2r}=\frac{1}{2} \left(3^{2 q+1}-1\right),\label{eq:Sum1} \\
		\sum _{r=0}^q 2^{2r}\binom{2q+1}{2r+1}=\frac{1}{4} \left(3^{2 q+1}+1\right).\label{eq:Sum2}
	\end{gather}
	Bearing in mind (\ref{eq:Sum1}) and (\ref{eq:Sum2}), it is convenient to rewrite (\ref{eq:DEvo2}) as
	\begin{equation}\label{eq:DEvo3}
		\binom{2k}{k} \left[ -4q\omega(0)+\sum _{r=1}^q 2^{2r}\omega(r) \left( \binom{2q+1}{2r}-2\binom{2q+1}{2r+1} \right) \right],
	\end{equation}
	where
	\[
		\omega(r):=2^{2r}\binom{2(k-r)}{k-r}\binom{2k}{k}^{-1} \: : \hspace{4mm} r=0,1,\ldots,q.
	\]
	Note that
	\[
		\frac{\omega(r+1)}{\omega(r)}=1+\frac{1}{2 (k- r)-1}>1 \: : \hspace{4mm} r=0,1,\ldots,q-1,
	\]
	i.e. $\omega(r)$ is a strictly increasing function of $r$. In order to determine which terms in (\ref{eq:DEvo3}) are negative and which are positive, we compute
	\[
		\frac{\binom{2q+1}{2r}}{2\binom{2q+1}{2r+1}}=-\frac{1}{2}+\frac{1+q}{1+2 (q- r)}.
	\]
	Therefore, defining $r_0:=\lfloor (4q+1)/6 \rfloor$, we see that the summand in (\ref{eq:DEvo3}) is strictly negative when $r\leq r_0$ and strictly positive when $r>r_0$. Since $\omega(r)$ is a strictly increasing function of $r$, it follows that the expression in (\ref{eq:DEvo3}) is strictly greater than
	\begin{multline*}
		\omega(r_0)\binom{2k}{k} \left[ -4q+\sum _{r=1}^q 2^{2r} \left( \binom{2q+1}{2r}-2\binom{2q+1}{2r+1} \right) \right] \\
		=\omega(r_0)\binom{2k}{k} \left[ 1+\sum _{r=0}^q 2^{2r} \left( \binom{2q+1}{2r}-2\binom{2q+1}{2r+1} \right) \right],
	\end{multline*}
	which, by (\ref{eq:Sum1}) and (\ref{eq:Sum2}), equals zero.
	
	Similarly, if $p$ is even, then, writing $p=2q$, (\ref{eq:T*Coefficient}) becomes
	\begin{equation}\label{eq:DEvo4}
		\sum _{r=0}^q 2^{4r}\binom{2(k-r)}{k-r}\binom{2q}{2r}-2\sum _{r=0}^{q-1} 2^{4r}\binom{2(k-r)}{k-r}\binom{2q}{2r+1} -\binom{2k}{k}.
	\end{equation}
	If $q=0$, then (\ref{eq:DEvo4}) vanishes. If $q=1$, then (\ref{eq:DEvo4}) equals
	\[
		16\binom{2(k-1)}{k-1}-4\binom{2k}{k}=4\binom{2k}{k}(\omega(1)-\omega(0))>0.
	\]
	Hence, assume $q\geq2$. Then we may express (\ref{eq:DEvo4}) as
	\begin{equation}\label{eq:DEvo5}
		\binom{2k}{k} \left[ -4q\omega(0)+\sum _{r=1}^{q-1} 2^{2r}\omega(r) \left( \binom{2q}{2r}-2\binom{2q}{2r+1} \right) +2^{2q}\omega(q) \right].
	\end{equation}
	Since
	\[
		\frac{\binom{2q}{2r}}{2\binom{2q}{2r+1}}=-\frac{1}{2}+\frac{1+2 q}{4 (q-r)},
	\]
	we see that, for $r_0:=\lfloor (4q-1)/6 \rfloor$, the summand in (\ref{eq:DEvo5}) is strictly negative when $r\leq r_0$ and strictly positive when $r>r_0$. It follows that (\ref{eq:DEvo5}) is strictly greater than
	\begin{multline}\label{eq:DEvo6}
		\omega(r_0)\binom{2k}{k} \left[ -4q+\sum _{r=1}^{q-1} 2^{2r} \left( \binom{2q}{2r}-2\binom{2q}{2r+1} \right) +2^{2q} \right] \\
		=\omega(r_0)\binom{2k}{k}\left[ -1+\sum_{r=0}^q2^{2r}\binom{2q}{2r}-2\sum_{r=0}^{q-1}2^{2r}\binom{2q}{2r+1} \right].
	\end{multline}
	Finally, since
	\begin{gather*}
		\sum _{r=0}^q 2^{2r}\binom{2q}{2r}=\frac{1}{2} \left(3^{2q}+1\right), \\
		\sum _{r=0}^{q-1} 2^{2r}\binom{2q}{2r+1}=\frac{1}{4} \left(3^{2q}-1\right),
	\end{gather*}
	we see that the expression given in (\ref{eq:DEvo6}) equals zero.
\end{proof}

\begin{cor}\label{cor:e2ke2l}
	Let $\mathcal{X}:=(a_1\pm ib_1,a_2\pm ib_2,\ldots,a_m\pm ib_m)$, where $a_i,b_i\geq0:i=1,2,\ldots,m$. Then for $1\leq k\leq l\leq m-1$,
	\[
		e_{2k}(\mathcal{X})e_{2l}(\mathcal{X})\geq \sqrt{\frac{l(m-k)}{k(m-l)}}e_{2k-1}(\mathcal{X})e_{2l+1}(\mathcal{X}).
	\]
\end{cor}
\begin{proof}
	If $k=l$, then the statement reduces to Theorem \ref{thm:e2k2}. If $k<l$, then by Theorems \ref{thm:e2k2} and \ref{thm:AllEven/AllOdd},
	\begin{align*}
		e_{2k}(\mathcal{X})e_{2l}(\mathcal{X}) &\geq \sqrt{e_{2k-1}(\mathcal{X})e_{2k+1}(\mathcal{X})e_{2l-1}(\mathcal{X})e_{2l+1}(\mathcal{X})} \\
		&\geq \sqrt{\frac{l(m-k)}{k(m-l)}}e_{2k-1}(\mathcal{X})e_{2l+1}(\mathcal{X}). \qedhere
	\end{align*}
\end{proof}

This completes the proof of Theorem \ref{thm:InequalitiesSummary}.

\section{Optimality and comparison to the generalised $\lambda$-Newton inequalities}\label{sec:Optimality}

In this section, we will show by example that Theorems \ref{thm:AllEven/AllOdd}, \ref{thm:DifferentParity} and \ref{thm:e2k2} are optimal. We will also compare our results to the corresponding generalised $\lambda$-Newton inequalities.

\begin{ex}
	Let us reconsider the list $\mathcal{X}:=(i,-i,i,-i,\ldots,i,-i)$ of length $2m$, given in Example \ref{ex:i}. We have $P_{2i}(\mathcal{X})=1$ for all $i=0,1,\ldots,m$, which gives equality in (\ref{eq:EvenStatement}).
\end{ex}

\begin{ex}
	Consider the list
	\[
		\mathcal{X}:=( \underbrace{i,-i,i,-i,\ldots,i,-i}_{m-1\: \mathrm{pairs}}\,,t,t),
	\]
	where $t$ is real. We have
	\begin{align}
		e_{2i}(\mathcal{X}) &= \binom{m-1}{i}+t^2\binom{m-1}{i-1} \: : \hspace{4mm} i=0,1,\ldots,m-1,\notag\\
		e_{2i+1}(\mathcal{X}) &= 2t\binom{m-1}{i} \: : \hspace{4mm} i=0,1,\ldots,m-1,\label{eq:eToP}\\
		e_{2m}(\mathcal{X}) &= t^2.\notag
	\end{align}
	Note that (\ref{eq:eToP}) is equivalent to $P_{2i+1}(\mathcal{X})=2t$: $i=0,1,\ldots,m-1$, which gives equality in (\ref{eq:OddStatement}). Furthermore, for all $1\leq k'\leq l'\leq m-1$,
	\[
		\lim_{t\rightarrow0}\frac{e_{2k'-1}(\mathcal{X})e_{2l'}(\mathcal{X})}{e_{2k'-2}(\mathcal{X})e_{2l'+1}(\mathcal{X})}=1
	\]
	and
	\[	
		\lim_{t\rightarrow\infty}\frac{e_{2k'}(\mathcal{X})e_{2l'+1}(\mathcal{X})}{e_{2k'-1}(\mathcal{X})e_{2l'+2}(\mathcal{X})}=1,
	\]
	which shows that (\ref{eq:EvenOddGap}) is optimal. Finally, if $t=\sqrt{\frac{m}{k}-1}$, then for all $1\leq k\leq m-1$,
	\[
		e_{2k}(\mathcal{X})^2=e_{2k-1}(\mathcal{X})e_{2k+1}(\mathcal{X})=4\binom{m-1}{k}^2,
	\]
	giving equality in (\ref{eq:e2k2}).
\end{ex}

Let us now compare the inequalities developed in Section \ref{sec:NewNewtonLikeInequalities} to the corresponding generalised $\lambda$-Newton inequalities (\ref{eq:GeneralisedLambdaNewtonInequalities}). Suppose, for example, that $\mathcal{X}$ consists of 8 complex-conjugate pairs. By Theorem \ref{thm:e2k2},
\begin{equation}\label{eq:e2k2Comp1}
	e_{8}(\mathcal{X})^2\geq e_7(\mathcal{X})e_9(\mathcal{X}).
\end{equation}
This is equivalent to
\[
	E_8(\mathcal{X})^2\geq \left( \frac{8}{9} \right)^2E_7(\mathcal{X})E_9(\mathcal{X}).
\]
Hence, if it is known that each $x_i$ lies in the wedge
\[
	\Omega=\{z\in\mathbb{C}:|\,\mathrm{arg}(z)\,|\leq\cos^{-1}(8/9)\},
\]
then the corresponding $\lambda$-Newton inequality is stronger than (\ref{eq:e2k2Comp1}). Otherwise, (\ref{eq:e2k2Comp1}) is stronger. This wedge is shown in Figure \ref{fig:Wedges} (left).

Note that as the values of $m$ and $k$ in Theorem \ref{thm:e2k2} grow larger, this critical wedge grows narrower. For example, if $\mathcal{X}$ consists of 100 complex-conjugate pairs, then
\[
	e_{100}(\mathcal{X})^2\geq e_{99}(\mathcal{X})e_{101}(\mathcal{X})
\]
is equivalent to
\[
	E_{100}(\mathcal{X})^2\geq \left( \frac{100}{101} \right)^2E_{99}(\mathcal{X})E_{101}(\mathcal{X}).
\]
The corresponding wedge is shown in Figure \ref{fig:Wedges} (right).

\begin{figure}[htb]
\centering
\begin{tikzpicture}
    \node[anchor=south west,inner sep=0] at (0,0) {\includegraphics[width=118mm]{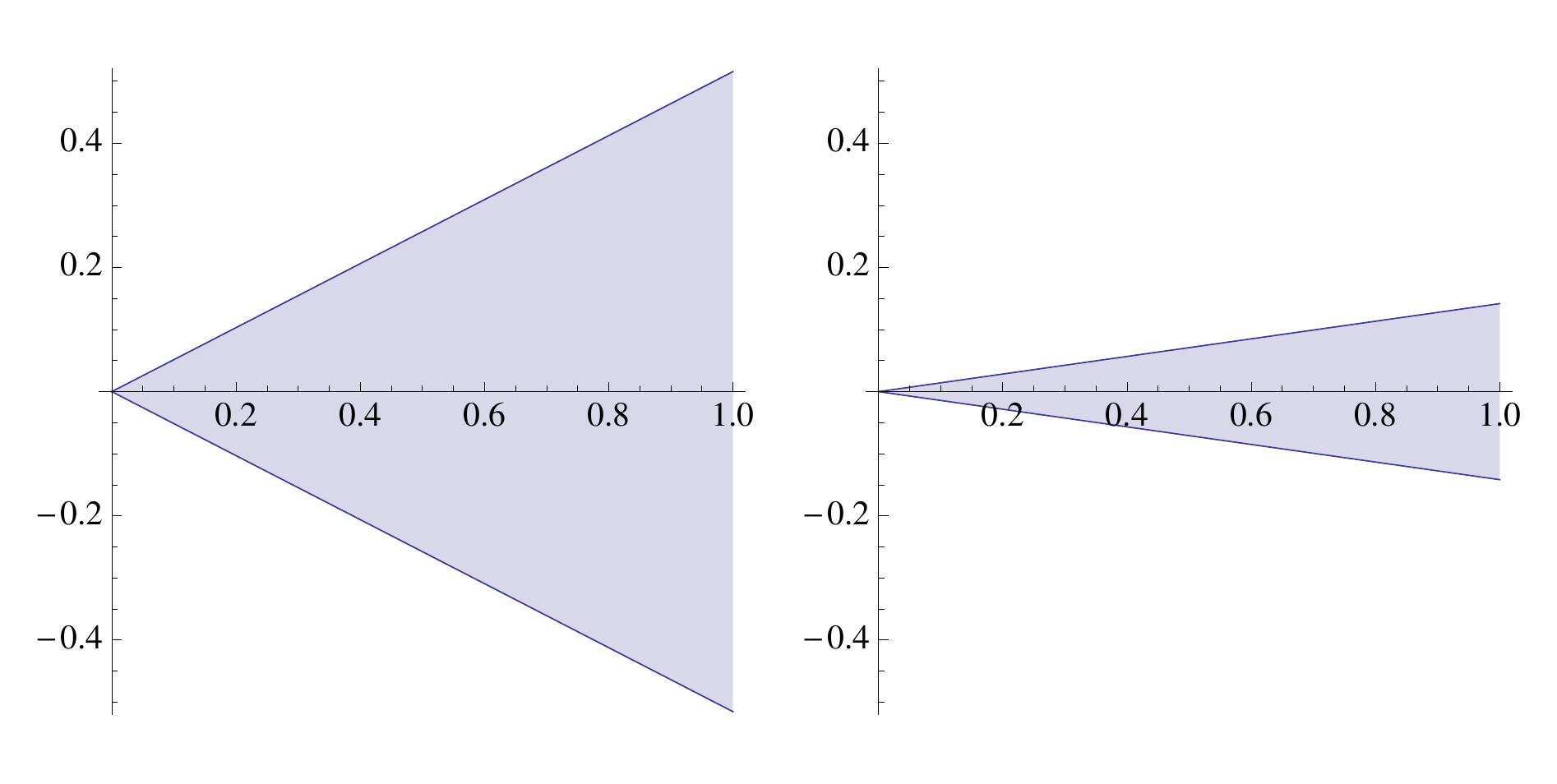}};
    \node (a) at (5.5,3.4) [] {\small$\mathrm{Re}(z)$};
    \node (b) at (0.9,5.7) [] {\small$\mathrm{Im}(z)$};
    \node (c) at (11.3,3.4) [] {\small$\mathrm{Re}(z)$};
    \node (d) at (6.7,5.7) [] {\small$\mathrm{Im}(z)$};
    
     \node (e) at (3.3,0) [] {\small$\{z\in\mathbb{C}:|\mathrm{arg}(z)|\leq\cos^{-1}(8/9)\}$};
     \node (f) at (9.1,0) [] {\small$\{z\in\mathbb{C}:|\mathrm{arg}(z)|\leq\cos^{-1}(100/101)\}$};
\end{tikzpicture}
\caption{Critical wedges}
\label{fig:Wedges}
\end{figure}
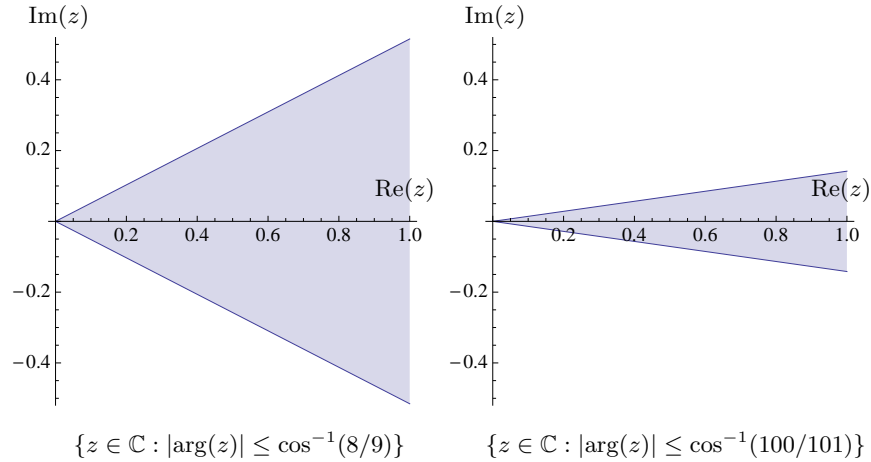

In general, it is clear that for any inequality given in Theorem \ref{thm:InequalitiesSummary}, there is a critical value of $\lambda$ (and an associated wedge $\Omega$) such that if each $x_i$ lies in $\Omega$, the associated generalised $\lambda$-Newton inequality gives a stronger result; however, if any of the $x_i$ lie outside of $\Omega$ (or the $x_i$ are unknown), then Theorem \ref{thm:InequalitiesSummary} will yield the stronger result. Furthermore, it is always possible to choose values of $k$, $l$ and $n$ in Theorem \ref{thm:InequalitiesSummary} such that this critical value of $\lambda$ is arbitrarily close to 1 and the corresponding wedge $\Omega$ is arbitrarily narrow. 

\bibliographystyle{amsplain}
\bibliography{Bibliography}

\end{document}